\documentclass[12pt]{article}
\usepackage{amsmath,amssymb,amsthm}
\usepackage{enumerate}
\usepackage{graphicx}
\usepackage[T1]{fontenc}
\usepackage{caption}
\usepackage{hyperref}
\usepackage{cite}

\def\draw #1 by #2 (#3){
	\vbox to #2{
		\hrule width #1 height 0pt depth 0pt
		\vfill
		\special{picture #3} 
	}
}

\def\scaleddraw #1 by #2 (#3 scaled #4){{
		\dimen0=#1 \dimen1=#2
		\divide\dimen0 by 1000 \multiply\dimen0 by #4
		\divide\dimen1 by 1000 \multiply\dimen1 by #4
		\draw \dimen0 by \dimen1 (#3 scaled #4)}
}

\newtheorem{theorem}{Theorem}[section]
\newtheorem{example}[theorem]{Example}
\newtheorem{problem}[theorem]{Problem}
\newtheorem{defin}[theorem]{Definition}
\newtheorem{lemma}[theorem]{Lemma}

\newtheorem{remark}[theorem]{Remark}
\usepackage{xcolor}
\usepackage{float}
\usepackage{graphicx,subcaption,lipsum}
\usepackage{hyperref}
\newtheorem{nt}{Note}

\newcommand{\singlespacing}{\let\CS=\@currsize\renewcommand{\baselinestretch}{1}\tiny\CS}
\newcommand{\oneandahalfspacing}{\let\CS=\@currsize\renewcommand{\baselinestretch}{1.25}\tiny\CS}
\newcommand{\doublespacing}{\let\CS=\@currsize\renewcommand{\baselinestretch}{1.35}\tiny\CS}

\newtheorem{rule-def}[theorem]{Rule}

\begin{document}
	\baselineskip 16pt
	\newcommand{\la}{\lambda}
	\newcommand{\si}{\sigma}
	\newcommand{\ol}{1-\lambda}
	\newcommand{\be}{\begin{equation}}
		\newcommand{\ee}{\end{equation}}
	\newcommand{\bea}{\begin{eqnarray}}
		\newcommand{\eea}{\end{eqnarray}}

	\baselineskip=0.30in
		\baselineskip=0.30in
         \begin{center}   {\Large\textbf{Extremal Graphs with Prescribed Parameters for the  Spectral Radius of Weighted Adjacency Matrices with Property $P^{*}$}  }
     	\\
     	\vspace{4mm}
        \end{center}
    \begin{center}
{\large \bf   	Swathi Shetty$^{1}$, B. R. Rakshith$^{*,2}$, Sayinath Udupa N. V. $^{3}$}\\
    	 Manipal Institute of Technology\\ Manipal Academy of Higher Education, Manipal, India\\
    	swathi.dscmpl2022@learner.manipal.edu$^{1}$\\
    	rakshith.br@manipal.edu; ranmsc08@yahoo.co.in$^{*,2}$\\
    	sayinath.udupa@manipal.edu	$^{3}$.
    \end{center}
    \footnotetext{*Corresponding author}
    \begin{center}(21 September 2026)\end{center}
        \begin{abstract}
        In the literature, several graph matrices based on vertex degrees have been introduced, and the study of their spectral properties has attracted considerable attention in recent years. Motivated by these developments, in this paper, we investigate the spectral radius of the weighted adjacency matrix $A_f(G)$, where the function $f$ satisfies property $P^{*}$. More precisely, we characterize the graphs that attain the maximum spectral radius of $A_f(G)$ among graphs with a vertex cut set of prescribed order,  graphs with exactly $s$ cut edges, and graphs with given vertex connectivity and independence number. Our results further develop the unified framework for studying extremal spectral properties of degree-based weighted adjacency matrices and extend this framework to several classes of graphs with prescribed structural parameters.
        \end{abstract}
        \textbf{Mathematics subject classification:} 05C50, 05C35.\\
\textbf{Keywords:} Weighted adjacency matrix, vertex connectivity, cut edges, independence number.
        \section{Introduction}

Let $G=(V,E)$ be a simple undirected graph, where $V=V(G)$ and $E=E(G)$ are the vertex set and edge set of $G$, respectively.   For two graphs $G_1=(V_1,E_1)$ and $G_2=(V_2,E_2)$, the disjoint union of $G_{1}$ and $G_{2}$, denoted by $G_1\cup G_2$, is the graph whose vertex set is the disjoint union of the sets $V_{1}$ and $V_{2}$, and whose edge set is $E(G_1)\cup E(G_2)$. The graph $kG$, where $k$ is a positive integer,  is the  disjoint union of $k$ copies of $G$. The join graph $G_1\vee G_2$ is obtained from $G_1\cup G_2$ by adding all possible edges between vertices of $G_1$ and $G_2$. For a subset $W$ of vertices, $G-W$ is the graph obtained from $G$ deleting all those vertices that are in $W$. Let $u$ and $v$ be two non-adjacent vertices of $G$. Then  $G+uv$ denotes the graph obtained from $G$ by adding an edge between $u$ and $v$. Suppose that $G$ is connected and that  $G-W$ is disconnected or an isolated vertex. Then $W$ is called a vertex cut set if $W\subseteq V(G)$. The cardinality of a minimum vertex cut set of $G$ is called the vertex connectivity of $G$, denoted by $\kappa(G)$. The vertex independence number of $G$, denoted by $\beta_{0}(G)$, is the maximum number of pair wise non-adjacent vertices in $G$. For any subset $S$ of $V(G)$, $G[S]$ denotes the subgraph of $G$ induced by $S$. As usual, the notations $K_{n}$ and $nK_{1}$ denote the complete graph and the complement of the complete graph on n vertices, respectively.\par

Topological indices are numerical quantities derived from the molecular graphs of chemical compounds and play an important role in mathematical chemistry. In recent years, degree–based topological indices have attracted considerable attention because of their strong predictive power in QSPR and QSAR studies~\cite{hasani2025modeling,das2025study,mondal2025role,das2025exponential,shetty2026extremal}. The study of degree-based topological indices has led to the introduction of new graph matrices, and research on these matrices has intensified in recent years~\cite{rather2026geometric,br2021zagreb,mondal2026degree,estrada2017abc,rakshith2025energy}. Let $V(G)=\{v_{1},v_{2},\ldots,v_{n}\}$ and $d_{i}$ be the degree of the vertex $v_{i}$.  Motivated by  studies on graph matrices derived from degree-based topological indices,  Das et al.~\cite{das2018degree} introduced the weighted adjacency matrix (extended adjacency matrix) of the graph $G$  as
\[
A_f(G)=
\begin{cases}
f(d_i,d_j), & \text{if } v_i\,\text{ is adjacent to}\, v_j\\
0, & \text{otherwise},
\end{cases}
\]
where $f$ is a non-negative real valued function of vertex degrees $d_{i}$ and $d_{j}$, satisfying the symmetry property $f(d_{i},d_{j})= f(d_{j},d_{i})$. Studies on $A_{f}(G)$ result in uniform
approaches for degree-based graph matrices, see \cite{cruz2022extremal,hu2022graphs,li2021trees,zheng2023extremal,li2022extremal} for more details.  The spectral radius of $A_{f}(G)$ is referred as the weighted spectral radius of $G$. A function $f=f(x,y)$ is said to be increasing and convex in $x$ if $f_x(x,y)\ge0$ and $f_{xx}(x,y)\ge0$.
Li and Wang~\cite{li2021trees} proved that when $f(x,y)$ is increasing and convex in $x$, the extremal tree of order $n$ that maximizes the weighted spectral radius is either the star graph $S_n$ or the double star $S_{d,n-d}$. Zheng et al.~\cite{zheng2023extremal} studied the spectral radius of $A_{f}(G)$ for  symmetric functions $f(x,y)$ satisfying the following conditions:
\begin{enumerate}
\item[(i)] $f(x,y)>0$ is increasing and convex in $x$;
\item[(ii)] if $x_1+y_1=x_2+y_2$ and $|x_1-y_1|>|x_2-y_2|$, then $f(x_1,y_1)\ge f(x_2,y_2)$.
\end{enumerate}
A function satisfying conditions $(i)$ and $(ii)$ is said to satisfy property $P^{*}$, and the corresponding matrix $A_f(G)$ is called a weighted adjacency matrix with property $P^{*}$. They characterized extremal trees and unicyclic graphs with respect to the spectral radius of $A_f(G)$.  \par
Motivated by these results, we investigate the spectral radius of the weighted adjacency matrix $A_f(G)$, where the function $f$ satisfies property $P^{*}$. More precisely, we characterize the graphs that attain the maximum weighted spectral radius among graphs with a vertex cut set of prescribed order,  graphs with exactly $s$ cut edges, and graphs with given vertex connectivity and independence number. These results contribute to a unified framework for characterizing extremal spectral properties of degree-based weighted adjacency matrices. 
\section{Main Results}
 We denote the spectral radius of a square matrix $M$ as $\rho(M)$. Let $\Omega_1$ and $\Omega_2$ be two real matrices of same order. If every entry in  $\Omega_{1}$ is less than or equal to the corresponding entry in $\Omega_{2}$, then we write $\Omega_1\preceq \Omega_2$. The weighted spectral radius of $G$ is denoted by $\rho_{f}(G)$. The following lemmas are important to prove our main results.  
  \begin{lemma}{\rm\cite{horn2012matrix}}\label{b_1b_2}
	Let $\Omega_1$ and $\Omega_2$ be non-negative real matrices of same order $n$. If $\Omega_1\preceq \Omega_2$, then $\rho(\Omega_1)\le \rho(\Omega_2)$. Further, if $\Omega_1$ is irreducible and $\Omega_1\neq \Omega_2$, then $\rho(\Omega_1)< \rho(\Omega_2)$.
\end{lemma}
\begin{remark}\label{addedge}
	 If $f=f(x,y)$ is an increasing function of $x$, then  $A_{f}(G)\preceq A_{f}(G+e)$. Therefore, by Lemma~\ref{b_1b_2},  $\rho(A_f(G))<\rho(A_{f}(G+e))\le \rho(K_{n})$.
\end{remark}

\begin{lemma}{\rm \cite{haemers1995interlacing}}\label{equitable}
Let $G$ be a graph and $\pi=(V_1,V_2,\dots,V_r)$ be a partition of $V(G)$ with quotient matrix $Q$. Then $\lambda_1(G)\ge \rho(Q)$, with equality if the partition is equitable.
\end{lemma}
Let $G$ be a connected graph with vertex set $V(G)=\{v_{1},v_{2},\ldots,v_{n}\}$. For two vertices $u$ and $v$ in $G$, define $G^{\prime}$ to be the graph obtained from $G$ by deleting every edge $v_{i}v$ $(1\le i\le n)$ such that $v_{i}\neq u$ and $v_{i}u\notin E(G)$,  and subsequently adding new edges $uv_{i}$. 
\begin{lemma}{\rm\cite{zheng2023extremal}}\label{zhengkel}
 If  a function $f$ satisfies property $P^\ast$ and $G\ncong G^{\prime}$, then $\rho(A_{f}(G))<\rho(A_{f}(G'))$.
\end{lemma}
In the subsequent subsections, the weighted spectral radius refers to the spectral radius of the matrix $A_f(G)$, where the function $f$ satisfies property $P^{\ast}$.
\subsection{ {Graphs with a vertex cut set of prescribed order}}
Let $\mathcal{V}_{n}^{k}$ be the set of all connected graphs of order $n$ with a vertex cut set of order $k$ ($1\le k\le n-1$). 
\begin{figure}[H]
\centering
\includegraphics[width=0.42\textwidth]{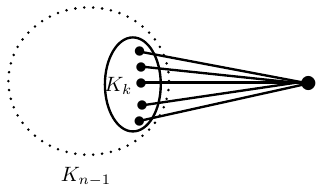}
\caption{The graph $K_{k}\vee \left(K_{n-1-k}\cup K_1\right)$.}
\label{knk}
\end{figure}
\begin{theorem}\label{at most k}
For any graph $G\in \mathcal{V}_{n}^{k}$, $$\rho_{f}(G)\le \rho_{f}(K_{k}\vee \left(K_{n-1-k}\cup K_1\right)),$$ where the equality holds if and only if $G\cong K_{k}\vee \left(K_{n-1-k}\cup K_1\right)$.
\end{theorem}
\begin{proof}
If $k=n-1$, then $\mathcal{V}_{n}^{k}$ is the set of all connected graphs of order $n$. By Remark \ref{addedge}, $\rho(G)\le \rho(K_{n})$ for any graph $G\in \mathcal{V}_{n}^{k}$. Thus the result holds for $k=n-1$. Otherwise, $k\le n-2$. Let $H$ be a graph in $\mathcal{V}_{n}^{k}$ with maximum weighted spectral radius, and let $W=\{u_{1},u_{2},\ldots,u_{k}\}$ be a vertex cut set of $H$ with $k$ vertices. Then the subgraph  $H\backslash W$, obtained by  removing all the vertices in $W$ has at least two components.\\[2mm]  	
\textbf{Claim 1.} $H\backslash W$ has exactly two components.\\
On contrary, suppose $H\backslash W$ has at least three components, say $H_1, H_2$ and $H_3$. Choose vertices $u\in H_1$ and $v\in H_2$, and add the edge $uv$ in $H$. Clearly, $W$ is a cut set of $H+uv$ with $k$ vertices, and hence $H+uv \in \mathcal{V}_{n}^{k}$. Moreover, by Remark \ref{addedge}, $\rho_f(H)<\rho_f(H+uv)$, which contradicts the assumption that $H$ has maximum weighted spectral radius in $\mathcal{V}_{n}^{k}$. Hence, $H\backslash W$ has exactly two components, say $H_1$ and $H_2$.\\[2mm]
\textbf{Claim 2.} The subgraph induced by the vertex sets $V(H_1)\cup W$  and $V(H_2)\cup W$ both are complete.\\
Let $i=1,2$. Suppose the vertex set $V(H_i)\cup W$ has a pair of non-adjacent vertices, say $u$ and $v$. Then the graph $H+uv\in \mathcal{V}_{n}^{k}$  as $W$ is also its vertex cut set. By Remark~\ref{addedge},  $\rho_f(H)<\rho_f(H+uv)$, a contradiction. Thus our claim is true.\\[2mm] By Claim 2, both $H_{1}$ and $H_{2}$ are cliques of order, say $n_{1}$ and $n_{2}$. That is, $H_{1}\cong K_{n_{1}}$ and $H_{2}\cong K_{n_{2}}$.\\[2mm]
 \textbf{Claim 3.} $n_1=1$ or $n_2=1$.\\
 Assume that $n_1>1$ and $n_2>1$. Then $deg_{H}(v)> |W|=k$ for any vertex $v$ of $H$.  Let $u\in H_{1}\, (i.e., K_{n_1})$ and $w\in H_{2}\, (i.e., K_{n_2})$. Consider the graph $H'$  obtained from $H$ by removing all edges incident with $u$ in $H_{1}$ and attaching those edges to the vertex $w$. Note that  $deg_{H^{\prime}}(u)=|W|=k$,  and thus $H\ncong H^{\prime}$. Therefore, by Lemma~\ref{zhengkel}, 
 $\rho_f(H)<\rho_{f}(H')$, which is a contradiction.\\[2mm]
 Hence $H\cong K_{k}\vee \left(K_{n-1-k}\cup K_1\right)$. This completes the proof.
\end{proof}
\subsection{Graphs with exactly $s$ cut edges}
We denote by $\mathcal{G}_{n}^{s}$ the set of all connected graphs of order $n$ with exactly $s$ cut edges.  \\[2mm]
Let $K_{1,s}$ be the star graph of order $s+1$ with vertex set $V(K_{1,s})=\left\{v_0,v_1,\dots,v_{s}\right\}$, where $v_0$ is its center. Let $K(n_0,n_1,n_2,\dots,n_{s})$ be the graph obtained by  $K_{1,s}$ by identifying each vertex $v_i$, $i=0,1,\dots,s$, with a vertex   of the clique $K_{n_i} (n_i\ge 1)$, see Fig.~\ref{kn0n1g}. Note that  $K\left(n-s,\underbrace{1,\dots,1}_{s}\right)=K_1\vee \left(K_{n-s-1}\cup s K_1\right)$. 
\begin{figure}[H]
\centering
\includegraphics[width=0.35\textwidth]{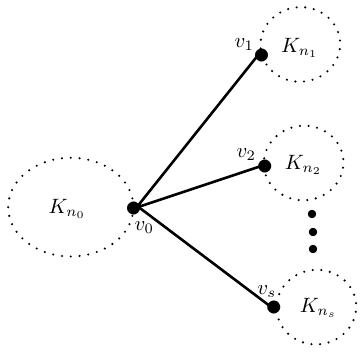}
\caption{The graph $K(n_0,n_1,n_2,\dots,n_{s})$.}
\label{kn0n1g}
\end{figure}
Let  $\mathcal{K}_{n}^{s}=\left\{K(n_0,n_1,n_2,\dots,n_{s}):n_i\ge 1,(0\le i\le s),\sum\limits_{i=0}^{s} n_i=n\right\}$.
\begin{theorem}
For any graph $G\in \mathcal{G}_{n}^{s}$, $$\rho_{f}(G)\le \rho_{f}\left(K_1\vee \left(K_{n-s-1}\cup s K_1\right)\right),$$ where the equality holds if and only if $G\cong K_1\vee \left(K_{n-s-1}\cup s K_1\right)$. 
\end{theorem}
\begin{proof}
For $s=0$,  $G\cong K_n$, and thus the result follows. Suppose $s\ge 1$.
Let $H$ be a graph in $\mathcal{G}_{n}^{s}$ with  maximum weighted spectral radius. Denote by  $H^\ast$ the graph obtained from $H$ by deleting $s$ cut edges. Then $H^{\ast}$ has $s+1$ components, each of which is either an isolated vertex or has order at least 3. Moreover, by Remark \ref{addedge}, each of these $s+1$ components must be a complete graph. Let $K_{n_{0}}, K_{n_{1}},\ldots,K_{n_{s}}$ be the components of $H^{\ast}$, where $\sum_{i=0}^{s}n_{i}=n$. For each $i$, define
$V_{n_i}=\{v\in V(K_{n_i}) : v \text{ is incident with a cut edge of } G\}$.\\[2mm]
\textbf{Claim 1.} $|V_{n_i}|=1$ for $0\le i\le s$.\\
Suppose $V_{n_{i}}$ has at least two vertices, say  $u$ and $w$. Construct  $H^{\prime}$ from $H$ by deleting every cut edge incident with $u$ and reattaching each such edge to $w$, with its other endpoint unchanged. Then $H^\prime\in \mathcal{G}_{n}^{s}$ and also, $H\not\cong H^\prime$ as the number of vertices of degree $n_i-1$ in $H$ is one less than that of $H^{\prime}$. Therefore, by Lemma~\ref{zhengkel},  $\rho_f(H^\prime)>\rho_f(H)$, a contradiction. Hence, $|V_{n_i}|=1$.\\[2mm]
From Claim 1, it follows that the graph $H$ must be isomorophic to a graph obtained from a tree $T$ with $V(T)=\{v_{0},v_{1},\ldots,v_{s}\}$ by identifying each vertex $v_{i}$ ($i = 0,1,2,\ldots,s$) of $T$ with a vertex of the clique $K_{n_{i}}$. For example, see Fig.~\ref{exg}.        
\begin{figure}[H]
\centering
\includegraphics[width=0.8\textwidth]{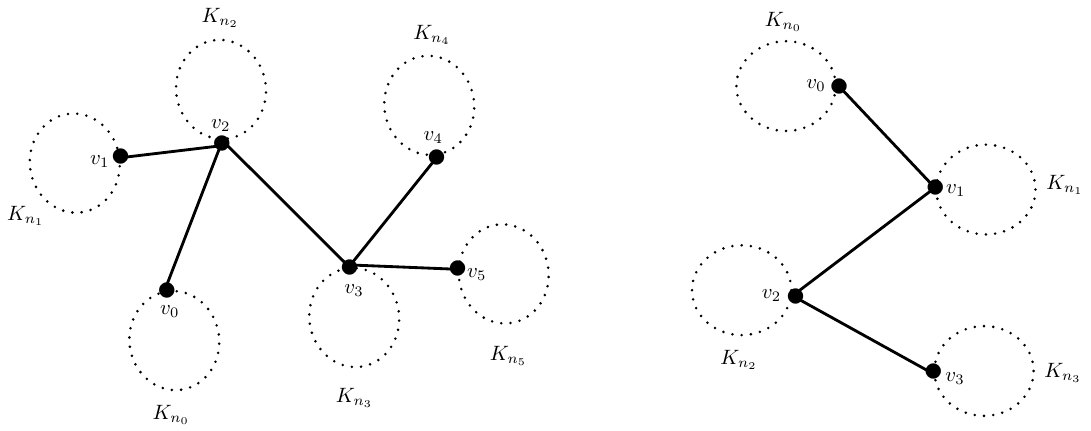}
\caption{ Examples of graphs $H$ arising from Claim 1. }
\label{exg}
\end{figure}
\noindent
\textbf{Claim 2.} $H\in \mathcal{K}_{n}^{s}$.\\
If $H\notin \mathcal{K}_{n}^{s}$, then there exist vertices $v\in V(K_{n_i})$ and $u\in V(K_{n_j})$ with $i\ne j$ such that $uv$ is a cut edge and each of $u$ and $v$ is incident with at least one other cut edge. Construct $H^{\prime}$ from $H$ by deleting every edge incident with $v$ except edge $uv$, and then reattaching each such edge to $u$, with its other endpoint unchanged. Then $H^\prime\in \mathcal{G}_{n}^{s}$. Moreover $H^{\prime}\ncong H$, since $H^{\prime}$ has one more pendant vertex than $H$.  By Lemma~\ref{zhengkel}, $\rho_f(H^\prime)>\rho_f(H)$, a contradiction. Hence, $H\in \mathcal{K}_{n}^{s}$.\\[2mm]
\textbf{Claim 3}. $H\cong K\left(n-s,\underbrace{1,\dots,1}_{s}\right)$.\\
By Claim 2, without loss of generality, we may assume that $H\cong K(n_{0},n_{1},n_{2},\ldots,n_{s})$. Suppose, to the contrary, that $n_{i}\ge 3$, for some $1\le i\le s$. Consider the graph $H^{\prime}$ obtained by deleting every edge incident with $v_{i}$ except the edge $v_{0}v_{i}$, and reattaching each such edge to $v_{0}$. Then  $H^{\prime}\in \mathcal{G}_{n}^{s}$. Since $H^{\prime}$ has one more pendant vertex than $H$, $H^{\prime}\ncong H$. Thus, by Lemma~\ref{zhengkel}, $\rho_f(H^\prime)>\rho_f(H)$, a contradiction. Therefore, $H\cong K\left(n_{0},\underbrace{1,\dots,1}_{s}\right)$, where $n_{0}=n-s$. That is,  $H\cong K_1\vee (K_{n-s-1}\cup sK_1)$. This completes the proof of the theorem.
\end{proof}

\subsection{Graphs with given vertex connectivity and independence number}
We denote by $\mathcal{H}_{n,\kappa}^{\beta_{0}}$  the set of graphs of order $n$ with connectivity $\kappa$ and independence number $\beta_0$.
\begin{figure}[H]
\centering
\includegraphics[width=0.48\textwidth]{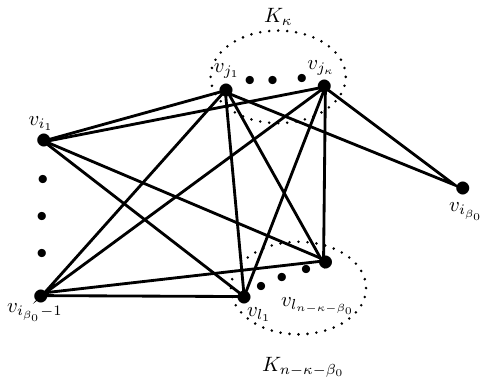}
\caption{The graph $K_{\kappa}\vee \left(K_1\cup \left(K_{n-\kappa-\beta_0}\vee(\beta_0-1)K_1\right)\right)$}
\label{Hnk}
\end{figure}
\begin{theorem}
For any graph $G\in \mathcal{H}_{n,\kappa}^{\beta_{0}}$,  $$\rho_{f}(G)\le \rho_{f}\left(K_{\kappa}\vee\left(K_1\cup \left(K_{n-\kappa-\beta_0}\vee (\beta_0-1)K_1\right)\right)\right),$$ where the equality holds if and only if $G\cong K_{\kappa}\vee\left(K_1\cup \left(K_{n-\kappa-\beta_0}\vee (\beta_0-1)K_1\right)\right)$, see Fig~\ref{Hnk}.
\end{theorem}
\begin{proof}
Let $H$ be a graph with vertex set $V(H)=\{v_1,v_2,\ldots ,v_n\}$ that attains the maximum weighted spectral radius among all graphs in $\mathcal{H}_{n,\kappa}^{\beta_0}$. 
Let $S=\{v_{i_1},\dots ,v_{i_{\beta_0}}\}$ be a maximum vertex independent set of $H$ and let $C=\{v_{j_1},\dots ,v_{j_\kappa}\}$ be a minimum vertex cut set of $H$, then the induced subgraph $H[V(G)\backslash C]$, denoted by $H^{\ast}$, is either disconnected graph or an isolated vertex. Let $H_1$, $H_2, \ldots, H_q$ be the components of $H^{\ast}$ with  
$|V(H_1)|\ge |V(H_2)|\ge \ldots \ge |V(H_q)|$. Note that, if $q=1$, then $H\cong K_n$, so we are done. Otherwise, $q\ge 2$.  Suppose $|C\cap S|=\alpha$, $|C\backslash S|=\gamma$, $|V(H_i)\cap S|=s_i$ and $|V(H_i)|\backslash S|=r_i$. \\[2mm]
\textbf{Claim 1}:  $H[V(H_i)\cup C]\cong(s_i+\alpha)K_1\vee K_{r_i+\gamma}$ for $i=1,2,\ldots,q$. \\[2mm]
Let $u$ and $v$ be two vertices of $H[V(H_i)\cup C]$, where at least one  of $u$ and $v$ does not belong to  $C$. Suppose that $uv\notin E(H)$. Then $H+uv\in \mathcal{H}_{n,\kappa}^{\beta_0}$. By Remark~\ref{addedge},  $\rho_f(H+uv)>\rho_f(H)$, which contradicts the hypothesis that $H$ has maximum weighted spectral radius. Thus, $uv\in E(H)$, and hence,  $H[V(H_i)\cup C]\cong(s_i+\alpha)K_1\vee K_{r_i+\gamma}$.\\[2mm]
Suppose $V(H)=C\cup S$. Then each  component $H_{i}$ is an isolated vertex. Furthermore, $C\cap S=\emptyset$, because if $C\cap S\neq\emptyset$, then $C\backslash S$ would form a cut set of $H$, contradicting the fact that the vertex connectivity of $H$ is $\kappa$. Thus, $H\cong K_{\kappa}\vee \beta_{0}K_{1}$.
\begin{figure}[H]
\centering
\includegraphics[width=0.43\textwidth]{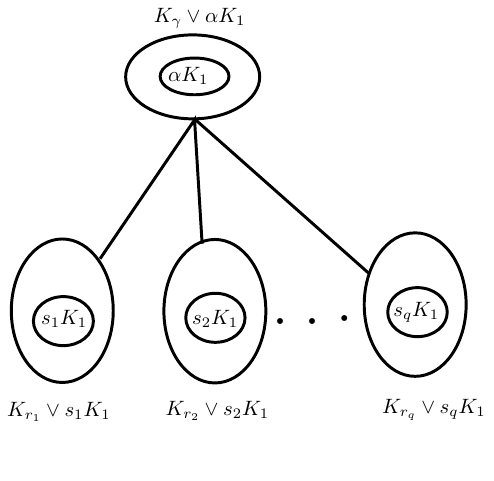}
\caption{Structure of graph $H$}
\label{fig511}
\end{figure}
\noindent
From Claim 1, it follows that the structure of the graph $H$ must be as depicted in Fig. \ref{fig511}. Let  $V(H)\neq C\cup S$. Then $H_{1}$ contains a vertex, say $w$, that belongs to neither $C$ nor $S$. We now claim that $q=2$. Suppose $q\ge 3$. Then for a vertex  $u$ in $H_{2}$, we have $H+uw \in \mathcal{H}_{n,\kappa}^{\beta_0}$. By Remark \ref{addedge}, it follows that  $\rho_{f}(H+uw)>\rho_{f}(H)$, a contradiction. Hence, $q=2$. Now, suppose  that  $V(H_{2})\ge 2$. Then $H_{2}$ contains a vertex, say $v$, that belongs to neither $C$ nor $S$. Consider the graph $H^{\prime}$ obtained by deleting all the edges in $H_{2}$ that are incident with vertex $v$, and reattaching each of these edges to the vertex $w$. Then, $H^{\prime}\in \mathcal{H}_{n,\kappa}^{\beta_0}$ and $H^{\prime}\ncong H$. Therefore, by Lemma \ref{zhengkel}, $\rho_{f}(H^{\prime})>\rho_{f}(H)$, a contradiction. Hence, $|V(H_{2})|=1$.\\[2mm]
\textbf{Claim 2}: $C\cap S=\emptyset$.\\
Assume contrary, suppose $x$ is a vertex common to both $C$ and $S$. Construct the graph $H^{\prime}$ by deleting all the edges that are incident with vertex $w$ and have the other end vertex in the set $S$, and then reattaching each of these edges to the vertex $x$. Note that, $deg_{H}(v)=deg_{H^{\prime}}(v)$ for all $v\in V(H)\backslash\{x,w\}$,  $deg_{H}(w)\le n-2$, $deg_{H}(x)\le n-2$, and $deg_{H}(x)= n-1$. So, the degree sequence of $H$ and $H^{\prime}$ are not same. Thus, $H^{\prime}\ncong H$. Therefore, by Lemma \ref{zhengkel}, $\rho_{f}(H^{\prime})>\rho_{f}(H)$, a contradiction.\\[2mm]
This completes the proof of the theorem.    
\end{proof}

\bibliography{name}
\bibliographystyle{abbrv}
\end{document}